\documentclass{amsart}[12pt]

\newtheorem{theorem}{Theorem}[section]
\newtheorem{lemma}[theorem]{Lemma}
\newtheorem{proposition}[theorem]{Proposition}
\newtheorem{corollary}[theorem]{Corollary}

\theoremstyle{definition}

\newtheorem{example}[theorem]{Example}

\newtheorem{remark}[theorem]{Remark}

\usepackage{amscd,amssymb}

\begin{document}

\title[Varieties of bicommutative algebras]
{Varieties of bicommutative algebras}
\author{Vesselin Drensky}
\date{}
\address{Institute of Mathematics and Informatics,
Bulgarian Academy of Sciences,
Acad. G. Bonchev Str., Block 8,
1113 Sofia, Bulgaria}
\email{drensky@math.bas.bg}
\thanks
{Partially supported by Grant I02/18
``Computational and Combinatorial Methods
in Algebra and Applications''
of the Bulgarian National Science Fund.}
\maketitle

\begin{abstract}
Bicommutative algebras are nonassociative algebras satisfying the polynomial identities
of right- and left-commutativity $(x_1x_2)x_3=(x_1x_3)x_2$ and $x_1(x_2x_3)=x_2(x_1x_3)$.
Let $\mathfrak B$ be the variety of all bicommutative algebras over a field $K$ of characteristic 0 and let $F({\mathfrak B})$ be the free
algebra of countable rank in $\mathfrak B$. We prove that if $\mathfrak V$ is a subvariety of $\mathfrak B$
satisfying a polynomial identity $f=0$ of degree $k$, where $0\not=f\in F({\mathfrak B})$, then the codimension sequence
$c_n({\mathfrak V})$, $n=1,2,\ldots$,
is bounded by a polynomial in $n$ of degree $k-1$. Since $c_n({\mathfrak B})=2^n-2$ for $n\geq 2$, and $\exp({\mathfrak B})=2$,
this gives that $\exp({\mathfrak V})\leq 1$, i.e., $\mathfrak B$ is minimal with respect to the codimension growth.
When the field $K$ is algebraically closed there are only three pairwise nonisomorphic two-dimensional bicommutative algebras $A$ which are nonassociative.
They are one-generated and with the property $\dim A^2=1$. We present bases of their polynomial identities and show that one of these algebras
generates the whole variety $\mathfrak B$.
\end{abstract}

\section{Introduction}

Bicommutative algebras are nonassociative algebras over a field $K$ satisfying the polynomial identities
of right- and left-commutativity
\begin{equation}\label{identities of bicommutativity}
(x_1x_2)x_3=(x_1x_3)x_2,\quad x_1(x_2x_3)=x_2(x_1x_3).
\end{equation}
In the sequel we consider algebras over a field $K$ of characteristic 0 only.
One-sided commutative algebras appeared first in the paper by Cayley \cite{Ca} in 1857.
In the modern language this is the right-symmetric Witt algebra $W_1^{\text{rsym}}$ in one variable.
Maybe the most important examples of one-side commutative algebras are Novikov algebras which are
left-commutative and right-symmetric. The latter means that the algebras satisfy the polynomial identity
$(x_1,x_2,x_3)=(x_1,x_3,x_2)$ for the associator $(x_1,x_2,x_3)=(x_1x_2)x_3-x_1(x_2x_3)$.
The motivation to study Novikov algebras comes from the needs of the Hamiltonian
operator in mechanics and the equations of hydrodynamics, see
Dzhumadil'daev, Ismailov, and Tulenbaev \cite{DIT} and Drensky and Zhakhayev \cite{DZ} for details.
By Kaygorodov and Volkov \cite{KaV} when the base field $K$ is algebraically closed of arbitrary characteristic
up to isomorphism there are only seven two-dimensional bicommutative algebras $A$ with nontrivial multiplication.
Four of them are associative-commutative.
Changing the notation and the bases of the algebras in \cite{KaV} the three nonassociative two-dimensional bicommutative algebras
\[
A_{\pi,\varrho}, \quad (\pi,\varrho)=(0,1), (1,0), (1,-1),
\]
are generated by one element $r$ and satisfy the condition $\dim A_{\pi,\varrho}^2=1$.
Their multiplication rules are
\begin{equation}\label{two-dimensional algebras}
rr^2=\pi r^2,r^2r=\varrho r^2,r^2r^2=\pi\varrho r^2.
\end{equation}
In is easy to see that the same holds over an arbitrary field $K$ of characteristic 0: Up to isomorphism
the three algebras $A_{0,1},A_{1,0},A_{1,-1}$, are the only
one-generated nonassociative two-dimensional bicommutative algebras $A$ with $\dim A^2=1$.

The structure of the free bicommutative algebra and the most important numerical invariants of the T-ideal of the polynomial identities
were described by Dzhumadil'daev, Ismailov, and Tulenbaev \cite{DIT}, see also the announcement \cite{DT}.
In \cite{DZ}, jointly with Zhakhayev, we proved that finitely generated bicommutative algebras are weakly noetherian, i.e., satisfy the ascending chain
condition for two-sided ideals, and answer into affirmative the finite basis problem for varieties of bicommutative algebras
over a field of arbitrary characteristic.

One of the most important measures for the complexity of the polynomial identities of a variety $\mathfrak V$ of $K$-algebras
is the codimension sequence $c_n({\mathfrak V})$,  $n=1,2,\ldots$, where
$c_n({\mathfrak V})$ is the dimension of the multilinear polynomials of degree $n$ in the free algebra $F_n({\mathfrak V})$ of rank $n$.
As a first approximation to the more precise estimate
of the growth of the codimensions one studies the behaviour of $\displaystyle \sqrt[n]{c_n({\mathfrak V})}$.
In the special case when
\[
\exp({\mathfrak V})=\lim_{n\to\infty}\sqrt[n]{c_n({\mathfrak V})}
\]
exists it is called the {\it exponent} of $\mathfrak V$, see Giambruno and Zaicev \cite{GZ1, GZ2}
who proved that for associative PI-algebras the exponent always exists and is an integer.
Following \cite{GZ3} the variety $\mathfrak V$ is {\it minimal of a given exponent} if
$\exp({\mathfrak W})<\exp({\mathfrak V})$ for all proper subvarieties $\mathfrak W$ of $\mathfrak V$.
(In \cite{D1} we called such varieties {\it extremal}.)

It was shown in \cite{DIT} that for the variety $\mathfrak B$ of all bicommutative algebras
\[
c_1({\mathfrak B})=1 \text{ and } c_n({\mathfrak B})=2^n-2,\quad n=2,3,\ldots.
\]
Hence $\exp({\mathfrak B})=2$. Our first main result is that the variety $\mathfrak B$ is minimal of exponent 2.
More precisely we show that if $\mathfrak V$ is a subvariety of $\mathfrak B$
satisfying a polynomial identity $f=0$ of degree $k$, where $0\not=f\in F({\mathfrak B})=F_{\infty}({\mathfrak B})$, then the codimension sequence
$c_n({\mathfrak V})$, $n=1,2,\ldots$,
is bounded by a polynomial in $n$ of degree $k-1$. The results of \cite{DIT} give that the variety $\mathfrak B$ is generated
by the free algebra $F_2({\mathfrak B})$ of rank 2. We slightly improve this and show that $\mathfrak B$ is generated
by the free algebra $F_1({\mathfrak B})$ of rank 1. As a byproduct of our approach, starting with the basis of  $F({\mathfrak B})$ in \cite{DIT}
we give a new proof of the description of the cocharacter sequence $\chi_n({\mathfrak B})$, $n=1,2,\ldots$.
Finally we study the polynomial identities of the two-dimensional algebras $A_{\pi,\varrho}$
with multiplication defined by (\ref{two-dimensional algebras}).
We show that the algebra $A_{1,-1}$ generates the whole variety $\mathfrak B$. The varieties
$\text{var}(A_{0,1})$ and $\text{var}(A_{1,0})$ generated respectively by the algebras $A_{0,1}$ and $A_{1,0}$
are defined as subvarieties of $\mathfrak B$ by the polynomial identities $x_1(x_2x_3)=0$ and $(x_1x_2)x_3=0$,
i.e., they are equal, respectively, to the varieties of left-nilpotent and right-nilpotent of class 3 bicommutative algebras.

\section{Preliminaries}

We fix a field $K$ of characteristic 0. All algebras, vector spaces, and tensor products will be over $K$.
Traditionally, one states the results on polynomial identities and cocharacter sequences in the language of representation theory
of the symmetric group $S_n$. Instead, we shall work with representation theory
of the general linear group $\text{\rm GL}_d=\text{\rm GL}_d(K)$. Then using the approach developed by Berele \cite{B} and the author \cite{D0}
we shall translate easily the results in terms of representations of $S_n$.
We start with the necessary background on representation theory
of $\text{\rm GL}_d$ acting canonically on the $d$-dimensional vector space $KX_d$
with basis $X_d=\{x_1,\ldots,x_d\}$.
We refer, e.g., to the books by Macdonald \cite{M} for general facts and by the author \cite{D2} for applications in the spirit of the problems
considered here.
All $\text{\rm GL}_d$-modules which appear in this paper are
completely reducible and are direct sums of irreducible polynomial
modules. The irreducible polynomial representations of $\text{\rm GL}_d$ are
indexed by partitions $\lambda=(\lambda_1,\ldots,\lambda_d)$,
$\lambda_1\geq \cdots\geq \lambda_d\geq 0$. We denote by
$W(\lambda)=W_d(\lambda)$ the corresponding irreducible $\text{\rm GL}_d$-module.
The action of $\text{\rm GL}_d$ on $KX_d$ is extended diagonally on the tensor algebra of $KX_d$
and, up to isomorphism, all $W(\lambda)$ can be found there. The tensor algebra of $KX_d$
is isomorphic, also as a $\text{\rm GL}_d$-module, to the free associative algebra $K\langle X_d\rangle=K\langle x_1,\ldots,x_d\rangle$.
Since the diagonal action of $\text{\rm GL}_d$ on the tensor algebra is not affected by the parentheses, we may work also in the absolutely free
algebra $K\{X_d\}$ and in the relatively free algebra $F_d({\mathfrak V})$ of any variety $\mathfrak V$.

The module $W(\lambda)\subset K\{X_d\}$ is generated by a
unique, up to a multiplicative constant,
multihomogeneous element $w_{\lambda}$ of degree $\lambda=(\lambda_1,\ldots,\lambda_d)$, i.e.,
homogeneous of degree $\lambda_k$ with respect to each variable $x_k$,
called the {\it highest weight vector} of $W(\lambda)$.
In order to state the characterization of the highest weight vectors we recall that for an algebra $R$
the linear operator $\delta:R\to R$ is a derivation
if $\delta(r_1r_2)=\delta(r_1)r_2+r_1\delta(r_2)$ for all $r_1,r_2\in R$.
If $\delta:KX_d\to KX_d$ is any linear operator of the $d$-dimensional vector space, then $\delta$ can be extended
in a unique way to a derivation of $K\langle X_d\rangle, K\{X_d\}$, and of any relatively algebra $F_d({\mathfrak V})$.
The following lemma is a partial case of a result by
De Concini, Eisenbud, and Procesi \cite{DEP}, see also
Almkvist, Dicks, and Formanek \cite{ADF}.
In the version which we need, the first part of the lemma was established
by Koshlukov \cite{K}.

\begin{lemma}\label{criterion for hwv} {\rm (see, e.g., Benanti and Drensky \cite{BD})}
Let $1\leq i<j\leq d$ and let
$\Delta_{x_j\to x_i}$ be the derivation of $K\{X_d\}$
defined by $\Delta_{x_j\to x_i}(x_j)=x_i$, $\Delta_{x_j\to x_i}(x_k)=0$, $k\not=j$.
If $w(X_d)=w(x_1,\ldots,x_d) \in K\{X_d\}$
is multihomogeneous of degree $\lambda_k$ with respect to $x_k$,
then $w(X_d)$ is a highest weight vector for some
$W(\lambda)$ if and only if $\Delta_{x_j\to x_i}(w(X_d))=0$
for all $i<j$. Equivalently, $w(X_d)$ is a highest weight vector
for $W(\lambda)$ if and only if
\[
g_{ij}(w(X_d))=w(X_d),\quad 1\leq i<j\leq d,
\]
where $g_{ij}$ is the linear operator of the $KX_d$ which
sends $x_j$ to $x_i+x_j$ and fixes the other $x_k$.
\end{lemma}

If $W_i$, $i=1,\ldots,m$,
are $m$ isomorphic copies of the $\text{\rm GL}_d$-module $W(\lambda)$
and $w_i\in W_i$ are highest weight
vectors, then the highest weight vector of any submodule $W(\lambda)$
of the direct sum $W_1\oplus\cdots\oplus W_m$ has the form
$\xi_1w_1+\cdots+\xi_mw_m$ for some $\xi_i\in K$.
Any $m$ linearly independent highest weight vectors can serve
as a set of generators of the $\text{\rm GL}_d$-module $W_1\oplus\cdots\oplus W_m$.
The algebra $F_d({\mathfrak V})$ decomposes as a $\text{\rm GL}_d$-module as
\begin{equation}\label{GL-decomposition of free algebras}
F_d({\mathfrak V})=\bigoplus_{\lambda}m_{\lambda}({\mathfrak V})W(\lambda),
\end{equation}
where the summation runs on all partitions $\lambda$ in not more than $d$ parts and the nonnegative integer $m_{\lambda}({\mathfrak V})$
is the multiplicity of $W(\lambda)$ in the decomposition of $F_d({\mathfrak V})$.
The canonical multigrading of $F_d({\mathfrak V})$ which counts the degree of each variable in $X_d$
agrees with the action of $\text{\rm GL}_d$ in the following way. Let
\[
H(F_d({\mathfrak V}),T_d)=H(F_d({\mathfrak V}),t_1,\ldots,t_d)
\]
\[
=\sum_{n_i\geq 0}\dim F_d^{(n)}({\mathfrak V})T_d^n
=\sum_{n_i\geq 0}\dim F_d^{(n)}({\mathfrak V})t_1^{n_1}\cdots t_d^{n_d}
\]
be the Hilbert series of $F_d({\mathfrak V})$ as a multigraded vector space, where $F_d^{(n)}({\mathfrak V})$ is the multihomogeneous
component of $F_d({\mathfrak V})$ of degree $n=(n_1,\ldots,n_d)$.
Then
\[
H(F_d({\mathfrak V}),T_d)=\sum_{\lambda}m_{\lambda}({\mathfrak V})s_{\lambda}(T_d)=\sum_{\lambda}m_{\lambda}({\mathfrak V})s_{\lambda}(t_1,\ldots,t_d),
\]
where $s_{\lambda}(T_d)$ is the Schur function corresponding to the partition $\lambda$.

There is another group action which is important for the theory of algebras with polynomial identities.
The symmetric group $S_n$ acts on the vector space $P_n({\mathfrak V})$ of the multilinear polynomials of degree $n$ in $F_n({\mathfrak V})$ by
\[
\sigma(f(x_1,\ldots,x_n))=f(x_{\sigma(1)},\ldots,x_{\sigma(n)}),\quad \sigma\in S_n,f\in P_n({\mathfrak V}).
\]
The $S_n$-character of $P_n({\mathfrak V})$ is called the $S_n$-{\it cocharacter} of $\mathfrak V$.
It is known that the decomposition of the $n$-th cocharacter
\begin{equation}\label{Sn-decomposition for cocharacters}
\chi_n({\mathfrak V})=\sum_{\lambda\vdash n}m_{\lambda}({\mathfrak V})\chi_{\lambda},
\end{equation}
where $\chi_{\lambda}$ is the irreducible $S_n$-character indexed with the partition $\lambda$ of $n$, is determined by the Hilbert series of
$F_n({\mathfrak V})$. The multiplicities $m_{\lambda}({\mathfrak V})$ are the same for $F_n({\mathfrak V})$ in (\ref{GL-decomposition of free algebras})
and for $\chi_n({\mathfrak V})$ in (\ref{Sn-decomposition for cocharacters}).
Finally, we recall a special case of the Young rule (and of the Littlewood-Richardson rule)
for the product of two Schur functions $s_{(p)}(T_d)$ and $s_{(q)}(T_d)$
(and also for the tensor product $W(p)\otimes W(q)$ of the $\text{\rm GL}_d$-modules $W(p)$ and $W(q)$). We assume that $p\geq q$. The case $p<q$
is similar.
\begin{equation}\label{Young rule}
\begin{split}
s_{(p)}(T_d)s_{(q)}(T_d)=\sum_{k=0}^qs_{(p+q-k,k)}(T_d),\\
W(p)\otimes W(q)\cong \bigoplus_{k=0}^qW(p+q-k,k).
\end{split}
\end{equation}
We shall need also estimates for the degree of the irreducible $S_n$-characters.

\begin{lemma}\label{degree for Sn-characters}
The degree $d_{\lambda}$ of the irreducible $S_n$-character $\chi_{\lambda}$, $\lambda=(\lambda_1,\lambda_2)\vdash n$,
is a polynomial in $n$ of degree $\lambda_2$.
\end{lemma}

\begin{proof} By the hook formula
\[
d_{\lambda}=\frac{n!}{\prod h_{ij}},
\]
where $h_{ij}$ is the length of the hook at the $(i,j)$-position of the Young diagram of $\lambda$.
For $\lambda=(\lambda_1,\lambda_2)\vdash n$ the lengths of the hooks of the first row are equal, reading them from right to left, to
\[
1,2,\ldots,n-2\lambda_2,n-2\lambda_2+2,\ldots,n-\lambda_2+1
\]
and those of the second row are $1,2,\ldots,\lambda_2$. Hence
\[
d_{\lambda}=\frac{n(n-1)\cdots (n-\lambda_2+1)(n-2\lambda_2+1)}{\lambda_2!},
\]
which is a polynomial of degree $\lambda_2$ in $n$.
\end{proof}

Let $\mathfrak B$ be the variety of all bicommutative algebras.
We assume that the free bicommutative algebras $F=F({\mathfrak B})$ and $F_d=F_d({\mathfrak B})$ are freely generated, respectively,
by the sets $X=\{x_1,x_2,\ldots\}$ and $X_d=\{x_1,\ldots,x_d\}$.
By \cite{DIT} the basis of the square $F^2_d$ of the algebra $F_d$ as a $K$-vector space consists of the following polynomials:
\begin{equation}\label{basis of F^2}
u_{i,j}=x_{i_1}(\cdots(x_{i_{p-1}}((\cdots((x_{i_p} x_{j_1})x_{j_2})\cdots)x_{j_q}))\cdots),
\end{equation}
where $p,q\geq 1$, $1\leq i_1\leq\cdots\leq i_{p-1}\leq i_p\leq d$, $1\leq j_1\leq j_2\leq\cdots\leq j_q\leq d$.
For any permutations $\sigma\in S_p$ and $\tau\in S_q$ the element $u_{i,j}$ from (\ref{basis of F^2}) satisfy the equality
\begin{equation}\label{action of Sm x Sn on the square}
u_{i,j}=x_{i_{\sigma(1)}}(\cdots(x_{i_{\sigma(p-1)}}((\cdots((x_{i_{\sigma(p)}} x_{j_{\tau(1)}})x_{j_{\tau(2)}})\cdots)x_{j_{\tau(q)}}))\cdots).
\end{equation}
The properties and the multiplication rules of $F_d({\mathfrak B})$ from \cite{DIT} are restated in \cite{DZ} in the following way.
We consider the polynomial algebra
\[
K[Y_d,Z_d]=K[y_1,\ldots,y_d,z_1,\ldots,z_d]
\]
in $2d$ commutative and associative variables.
We associate to each monomial $u_{i,j}$ in (\ref{basis of F^2}) the monomial
\[
\psi(u_{i,j})=y_{i_1}\cdots y_{i_{p-1}}y_{i_p} z_{j_1}z_{j_2}\cdots z_{j_q}\in K[Y_d,Z_d]
\]
and extend $\psi$ by linearity to a linear map $\psi:F_d^2\to K[Y_d,Z_d]$. The image $\psi(F_d^2)$ is spanned by all monomials
\[
Y_d^{\alpha}Z_d^{\beta}=y_1^{\alpha_1}\cdots y_d^{\alpha_d}z_1^{\beta_1}\cdots z_d^{\beta_d},\quad
\vert\alpha\vert=\sum_{i=1}^d\alpha_i>0,\vert\beta\vert=\sum_{j=1}^d\beta_j>0.
\]
Then we define an algebra $G_d$ generated by $X_d$ with basis
\[
X_d\cup\{Y_d^{\alpha}Z_d^{\beta}\mid\vert\alpha\vert,\vert\beta\vert>0\}
\]
and multiplication rules
\begin{equation}\label{multiplication in G_d}
\begin{array}{c}
x_ix_j=y_iz_j,\\
\\
x_i(Y_d^{\alpha}Z_d^{\beta})=y_iY_d^{\alpha}Z_d^{\beta},\\
\\
(Y_d^{\alpha}Z_d^{\beta})x_j=Y_d^{\alpha}Z_d^{\beta}z_j,\\
\\
(Y_d^{\alpha}Z_d^{\beta})(Y_d^{\gamma}Z_d^{\delta})=Y_d^{\alpha+\gamma}Z_d^{\beta+\delta}.\\
\end{array}
\end{equation}
The algebras $F_d$ and $G_d$ are isomorphic both as algebras and as multigraded vector spaces with isomorphism $\psi:F_d\to G_d$
which sends $x_i\in F_d$ to $x_i\in G_d$ and acts on $F_d^2$ in the same way as the linear map
$\psi:F_d^2\to K[Y_d,Z_d]$ defined above.

\section{Free bicommutative algebras}
In this section we give an alternative proof of the formula for the cocharacter sequence of $\mathfrak B$ given in \cite{DIT}
and describe the highest weight vectors of the irreducible $\text{\rm GL}_d$-submodules of $F_d=F_d({\mathfrak B})$.
The action of $\text{\rm GL}_d$ on the $d$-dimensional vector space $KX_d$ induces a similar action on $KY_d$ and $KZ_d$ which is extended diagonally
on the polynomial algebras $K[Y_d]$ and $K[Z_d]$ and on the square $G_d^2$ of the algebra $G_d$.
In the sequel we equip $K[Y_d]$, $K[Z_d]$, and $G_d^2$ with this action of $\text{\rm GL}_d$.

\begin{lemma}\label{the graded structure of square of F}
As a multigraded vector space the square $F_d^2$ of the free bicommutative algebra $F_d$ is isomorphic to the tensor product
$\omega(K[Y_d])\otimes\omega(K[Z_d])$, where $\omega$ is the augmentation ideal of the polynomial algebra, i.e., the ideal of polynomials
without constant term. As a $\text{\rm GL}_d$-module $F_d^2$ is isomorphic to the direct sum of tensor products
\begin{equation}\label{GL-structure of the square of F}
\bigoplus_{p,q\geq 1}W(p)\otimes W(q).
\end{equation}
\end{lemma}

\begin{proof}
We identify the monomial $Y_d^{\alpha}Z_d^{\beta}\in K[Y_d,Z_d]$ with $Y_d^{\alpha}\otimes Z_d^{\beta}\in K[Y_d]\otimes K[Z_d]$.
Then the first part of the lemma is simply a restatement of the fact that the image of $F_d^2$ under the action of $\psi$ has a basis
$\{Y_d^{\alpha}Z_d^{\beta}\mid\vert\alpha\vert,\vert\beta\vert>0\}$.  The second part of the lemma holds because
the $\text{\rm GL}_d$-module $K[Y_d]^{(p)}$ of the homogeneous polynomials of degree $p$ in $K[Y_d]$ is isomorphic to $W(p)$ and similarly for $K[Z_d]^{(q)}$.
\end{proof}

\begin{proposition}\label{the n-th cocharacter of B}\cite{DIT}
The cocharacter sequence of the variety $\mathfrak B$ of all bicommutative algebras is
\[
\chi_n({\mathfrak B})=\sum_{(\lambda_1,\lambda_2)\vdash n}m_{(\lambda_1,\lambda_2)}({\mathfrak B})\chi_{(\lambda_1,\lambda_2)},
\]
where
\begin{equation}\label{multiplicities of cocharacters of B}
\begin{split}
m_{(1)}({\mathfrak B})=1,\\
m_{(n)}({\mathfrak B})=n-1,\quad n>1,\\
m_{(\lambda_1,\lambda_2)}({\mathfrak B})=n-2\lambda_2+1,\quad \lambda_2>0.
\end{split}
\end{equation}
\end{proposition}

\begin{proof}
The multiplicities of the irreducible $S_n$-characters in the cocharacter sequence (\ref{Sn-decomposition for cocharacters})
and of the irreducible $\text{\rm GL}_d$-modules of the homogeneous component $F_d^{(n)}$ of degree $n$ of the free algebra $F_d$ in (\ref{GL-decomposition of free algebras})
are the same for $d\geq n$. Hence we may work in $F_d$ instead of with $\chi_n({\mathfrak B})$. Since the case $n=1$ is trivial, we shall assume that $n>1$.
By Lemma \ref{the graded structure of square of F} and the Young rule (\ref{Young rule}) we derive that the only nontrivial multiplicities
$m_{\lambda}({\mathfrak B})$ are for $\lambda=(\lambda_1,\lambda_2)$. Then $m_{\lambda}({\mathfrak B})$ is equal to the number of tensor products
$W(p)\otimes W(q)$ which contain an isomorphic copy of $W(\lambda)$ as a submodule. For $\lambda=(n)$ there are $n-1$ possibilities
\[
W(1)\otimes W(n-1),W(2)\otimes W(n-2),\ldots,W(n-1)\otimes W(1),
\]
i.e., $m_{(n)}({\mathfrak B})=n-1$. For $\lambda=(\lambda_1,\lambda_2)$ with $\lambda_2>0$ the possibilities are
\[
W(\lambda_2)\otimes W(n-\lambda_2),W(\lambda_2+1)\otimes W(n-\lambda_2-1),\ldots,W(n-\lambda_2)\otimes W(\lambda_2),
\]
which gives $m_{(\lambda_1,\lambda_2)}({\mathfrak B})=n-2\lambda_2+1$.
\end{proof}

\begin{lemma}\label{hwv of G}
The following polynomials $w_{\lambda}^{(k)}$ form a maximal linearly independent system of highest weight vectors of the $\text{\rm GL}_d$-submodules $W(\lambda)$ in $G_d^2$:
\begin{equation}\label{hwv of W(lambda)}
\begin{split}
w_{(n)}^{(j)}=y_1^jz_1^{n-j},&\quad j=1,2,\ldots,n-1,\\
w_{\lambda}^{(j)}=y_1^j(y_1z_2-y_2z_1)^{\lambda_2}z_1^{\lambda_1-\lambda_2-j}, &\quad j=0,1,\ldots,\lambda_1-\lambda_2,\text{ if }\lambda_2>0.
\end{split}
\end{equation}
\end{lemma}

\begin{proof}
For a fixed $\lambda$ the elements (\ref{hwv of W(lambda)}) are linearly independent because are nonzero and of pairwise different degree in $y_1$.
They are of degree $\lambda_1$ with respect to $y_1,z_1$ and of degree $\lambda_2$ with respect to $y_2,z_2$.
By Proposition \ref{the n-th cocharacter of B} the multiplicities of $W(n)$ and
$W(\lambda)$, $\lambda=(\lambda_1,\lambda_2)\vdash n$, in $G_d^2$
are, respectively,
\[
m_{(n)}({\mathfrak B})=n-1\text{ and }m_{(\lambda_1,\lambda_2)}({\mathfrak B})=n-2\lambda_2+1=\lambda_1-\lambda_2+1.
\]
Hence their number coincides with the number of polynomials in (\ref{hwv of W(lambda)}).
Now, it is sufficient to show that all $w_{\lambda}^{(j)}$ are highest weight vectors. Applying Lemma \ref{criterion for hwv},
this is obvious for $w_{(n)}^{(j)}$. Let $\lambda_2>0$. The analogue $\Delta_{y_2\to y_1,z_2\to z_1}$ of the derivation $\Delta_{x_2\to x_1}$ acting on $K[Y_d,Z_d]$
sends $y_1,z_1$ to 0 and $y_2,z_2$ to $y_1,z_1$, respectively.
Obviously
\[
\Delta_{y_2\to y_1,z_2\to z_1}(w_{\lambda}^{(j)})
=\lambda_2y_1^j(y_1z_2-y_2z_1)^{\lambda_2-1}z_1^{\lambda_1-\lambda_2-j}\Delta_{y_2\to y_1,z_2\to z_1}(y_1z_2-y_2z_1)=0
\]
and all $w_{\lambda}^{(j)}$ are highest weight vectors.
\end{proof}

\section{Subvarieties}

In this section we assume that $\mathfrak V$ is a proper subvariety of $\mathfrak B$ and $\mathfrak V$ satisfies a nontrivial polynomial identity
$f=0$ of degree $k$, where $0\not=f(X_d)\in F=F({\mathfrak B})$. Since the case $k=1$ is trivial we shall assume that $k\geq 2$.
In the sequel we shall work mainly in the isomorphic copies $G$ and $G_d$ of the algebras $F$ and $F_d$ instead of in $F$ and $F_d$.
Identifying $F$ and $F_d$ with their isomorphic copies, we shall denote the corresponding elements with the same symbols.
In particular, if $f(X_d)\in F_d^2$ we shall write $f(Y_d,Z_d)\in G_d^2$ and vise versa.
Since the $\text{\rm GL}_d$-module generated by $f$ contains an irreducible submodule $W(\lambda)$,
there exists a highest weight vector $w_{\lambda}$ such that the polynomial identity $w_{\lambda}=0$ follows from the polynomial identity $f=0$.
Hence we may assume that $\mathfrak V$ satisfies some polynomial identity $w_{\lambda}(x_1,x_2)=0$ for $\lambda\vdash k$.
Then $w_{\lambda}(Y_2,Z_2)\in G_2$ is a linear combination of the highest weight vectors in (\ref{hwv of W(lambda)}) and for some $\xi_j\in K$
\begin{equation}\label{hwv for subvariety}
\begin{split}
w_{(k)}=\sum_{j=1}^{k-1}\xi_jy_1^jz_1^{k-j},&\text{ for } \lambda=(k),\\
w_{(k)}=(y_1z_2-y_2z_1)^{\lambda_2}\sum_{j=0}^{\lambda_1-\lambda_2}\xi_jy_1^jz_1^{\lambda_1-\lambda_2-j},
&\text{ for } \lambda=(\lambda_1,\lambda_2),\lambda_2>0.
\end{split}
\end{equation}
If $f(X_d)\in F_d$ is multihomogeneous then its partial linearization
$\text{lin}_{x_i}f(X_d)$ in $x_i$ is the component of degree $\deg_{x_i}-1$ with respect to $x_i$ of the polynomial
$f(x_1,\ldots,x_i+x_{d+1},\ldots,x_d)\in F_{d+1}$. If $\Delta_{x_i\to x_{d+1}}$ is the derivation of $F_{d+1}$
which sends $x_i$ to $x_{d+1}$ and the other $x_j$ to 0,
then
\[
\text{lin}_{x_i}f(X_d)=(\text{lin}_{x_i}f)(X_{d+1})=\Delta_{x_i\to x_{d+1}}(f(X_d)).
\]
If $u\in F$ then $(\text{lin}_{x_i}f)(x_1,\ldots,x_d,u)$ can be expressed in terms of derivations as
\[
(\text{lin}_{x_i}f)(x_1,\ldots,x_d,u)=\Delta_{x_i\to u}(f(X_d)),
\]
where $\Delta_{x_i\to u}$ is the derivation of $F$ sending $x_i$ to $u$ and all other generators $x_j$ to 0.
The action of the analogue of $\Delta_{x_i\to x_{d+1}}$ on $G^2$ is clear: It sends $y_i$ and $z_i$, respectively, to $y_{d+1}$ and $z_{d+1}$ and all other
variables $y_j$ and $z_j$ to 0. We denote this derivation by $\Delta_{y_i\to y_{d+1},z_i\to z_{d+1}}$.
Now we shall translate the action of $\delta_{x_i\to u}$ on $F^2$, $u\in F^2$,
in the language of $G$ and the usual partial derivatives.

\begin{lemma}\label{action of delta}
Let $u\in K[Y,Z]$ be in the image $G^2\subset K[Y,Z]$ of $F^2$.
Let $\Delta_{y_i,z_i\to u}$ be the derivation of $K[Y,Z]$ which sends the variables $y_i,z_i$ to $u$ and the other variables to $0$.
If $f(X_d)\in F^2$ is multihomogeneous, then the image of
$(\text{lin}_{x_i}f)(x_1,\ldots,x_d,u)$ in $G^2$ is
\[
\Delta_{y_i,z_i\to u}(f)=\left(\frac{\partial f}{\partial y_i}+\frac{\partial f}{\partial z_i}\right)u.
\]
\end{lemma}

\begin{proof}
It is sufficient to consider the case when $f$ and $u$ are monomials and $i=1$:
\[
f=(y_1^{\alpha_1}z_1^{\beta_1})v_1v_2, v_1=y_2^{\alpha_2}\cdots y_d^{\alpha_d},v_2=z_2^{\beta_2}\cdots z_d^{\beta_d},
\]
\[
\alpha_1+\beta_1\geq 1,\vert\alpha\vert=\alpha_1+\alpha_2+\cdots+\alpha_d>0,\vert\beta\vert=\beta_1+\beta_2+\cdots+\beta_d>0,
\]
\[
u=Y_d^{\gamma}Z_d^{\delta},\vert\gamma\vert>0,\vert\delta\vert>0.
\]
Then
\[
\Delta_{y_1\to y_{d+1},z_1\to z_{d+1}}(f)=(\alpha_1y_1^{\alpha_1-1}y_{d+1}z_1^{\beta_1}+\beta_1y_1^{\alpha_1}z_1^{\beta_1-1}z_{d+1})v_1v_2
\]
\[
=\frac{\partial f}{\partial y_1}y_{d+1}+\frac{\partial f}{\partial z_1}z_{d+1},
\]
In virtue of (\ref{action of Sm x Sn on the square}) we may assume that the preimage
$\displaystyle \psi^{-1}\left(\frac{\partial f}{\partial y_1}y_{d+1}\right)$ in $F_{d+1}^2$ is of the form
$\alpha_1(\cdots(x_{d+1}x_{j_1})\cdots)$, where the dots before and after $(x_{d+1}x_{j_1})$ correspond to the beginning and the end of the element in the form
(\ref{basis of F^2}). Since $u\in F^2$ we obtain that
\[
\psi(\alpha_1(\cdots(ux_{j_1})\cdots))=\alpha_1(\cdots (Y_d^{\gamma}Z_d^{\delta}z_{j_1})\cdots)=\frac{\partial f}{\partial y_1}u.
\]
Similarly
\[
\psi(\beta_1(\cdots(x_{i_p}u)\cdots))=\beta_1(\cdots (y_{i_p}Y_d^{\gamma}Z_d^{\delta})\cdots)=\frac{\partial f}{\partial z_1}u.
\]
\end{proof}

\begin{lemma}\label{consequences of w(lambda)}
If $0\not=f\in W(\lambda_1,\lambda_2)\subset F({\mathfrak B})$, then all polynomial identities $w_{(\mu_1,\mu_2)}^{(j)}=0$ with $\mu_2\geq \lambda_1$
are consequences of the polynomial identity $f=0$.
\end{lemma}

\begin{proof}
As commented in the beginning of the section, we may assume that $f=w_{\lambda}$ is a highest weight vector in $W(\lambda_1,\lambda_2)\subset F_2$.
Hence, working in $G_2$ instead of in $F_2$, $w_{\lambda}$ has the form (\ref{hwv for subvariety}), i.e.,
\[
w_{\lambda}=\sum_{j\geq p}\xi_j w_{\lambda}^{(j)}=(y_1z_2-y_2z_1)^{\lambda_2}\sum_{j\geq p}\xi_jy_1^jz_1^{\lambda_1-\lambda_2-j},\xi_p\not=0.
\]
First, let $p>0$, i.e., $w_{\lambda}$ is divisible by $y_1^p$.
The partial linearizations of the identity $w_{\lambda}=0$ are its consequences. Hence $\Delta_{y_1\to y_2,z_1\to z_2}(w_{\lambda})=0$ which has the form
\[
(y_1z_2-y_2z_1)^{\lambda_2}\sum_{j\geq p}\xi_j(jy_1^{j-1}y_2z_1^{\lambda_1-\lambda_2-j}+(\lambda_1-\lambda_2-j)y_1^jz_1^{\lambda_1-\lambda_2-j-1}z_2)=0
\]
is also a consequence of $w_{\lambda}=0$ and the same holds for
\[
w_{(\lambda_1,\lambda_2+1)}=\Delta_{y_1\to y_2,z_1\to z_2}(w_{\lambda})z_1-(\lambda_1-\lambda_2)w_{\lambda}z_2
\]
\[
=-(y_1z_2-y_2z_1)^{\lambda_2+1}\sum_{j\geq p-1}(j+1)\xi_{j+1}y_1^jz_1^{\lambda_1-\lambda_2-j-1}=0.
\]
We obtained that $w_{(\lambda_1,\lambda_2+1)}=0$ is a consequence of $w_{(\lambda_1,\lambda_2)}=0$.
It is divisible by $y_1^{p-1}$ but is not divisible by $y_1^p$.
Continuing in this way we shall reach a consequence
\[
w_{(\lambda_1,\lambda_2+p)}=(y_1z_2-y_2z_1)^{\lambda_2+p}j!\xi_pz_1^{\lambda_1-\lambda_2-p}=0.
\]
Now the consequence
\[
y_1\Delta_{y_1\to y_2,z_1\to z_2}(w_{(\lambda_1,\lambda_2+p)})-y_2(\lambda_1-\lambda_2+p)w_{(\lambda_1,\lambda_2+p)}=0
\]
is of the form $w_{(\lambda_1,\lambda_2+p+1)}=0$ and is divisible by $z_1^{\lambda_1-\lambda_2-p-1}$ only. Continuing the process we shall obtain
as a consequence
\[
w_{(\lambda_1,\lambda_1)}=(y_1z_2-y_2z_1)^{\lambda_1}=w_{(\lambda_1,\lambda_1)}^{(0)}.
\]
Since all $w_{(\mu_1,\mu_2)}^{(j)}$ with $\mu_2\geq\lambda_1$ are divisible by $w_{(\lambda_1,\lambda_1)}^{(0)}$
and hence are its consequences, we complete the proof.
\end{proof}

\begin{corollary}\label{consequences of identity of degree k}
If $0\not=f\in F$ is of degree $k$ then all identities $w_{(\mu_1,\mu_2)}^{(j)}=0$ with $\mu_2\geq k$ follow from the identity $f=0$.
\end{corollary}

\begin{proof}
The statement follows immediately from Lemma \ref{consequences of w(lambda)} because if $(\lambda_1,\lambda_2)\vdash k$, then $\lambda_1\leq k$.
\end{proof}

\begin{lemma}\label{consequences in one variable}
The polynomial identity $w_{(k,k)}^{(0)}=(y_1z_2-y_2z_1)^k=0$ has as consequences all identities
\[
(y_1z_1)^k(y_1-z_1)^kw_{\mu}^{(j)}=0
\]
for all $\mu=(\mu_1,\mu_2)$ and all $j=0,1,\ldots,\mu_1-\mu_2$.
\end{lemma}

\begin{proof}
We apply the derivation $\Delta_{y_2,z_2\to y_1z_1}$ and obtain as a consequence of the identity $w_{(k,k)}^{(0)}=0$ the identity
\[
\Delta_{y_2,z_2\to y_1z_1}(w_{(k,k)}^{(0)})=y_1z_1\left(\frac{\partial}{\partial y_2}+\frac{\partial}{\partial z_2}\right)(y_1z_2-y_2z_1)^k
\]
\[
=ky_1z_1(y_1z_2-y_2z_1)^{k-1}\left(\frac{\partial}{\partial y_2}+\frac{\partial}{\partial z_2}\right)(y_1z_2-y_2z_1)
\]
\[
=ky_1z_1(y_1-z_1)(y_1z_2-y_2z_1)^{k-1}=0.
\]
Continuing in this way we obtain
\[
\Delta_{y_2,z_2\to y_1z_1}^k(w_{(k,k)}^{(0)})=k!(y_1z_1)^k(y_1-z_1)^k=0
\]
which gives that $(y_1z_1)^k(y_1-z_1)^kw_{\mu}^{(j)}=0$ for all $\mu$ and all $j$.
\end{proof}

\begin{corollary}\label{generating by free algebra of rank 1}
The variety $\mathfrak B$ is generated by its one-generated free algebra $F_1({\mathfrak B})$.
\end{corollary}

\begin{proof}
If $\text{var}(F_1({\mathfrak B}))\not=\mathfrak B$, then by Lemma \ref{consequences of w(lambda)} the algebra
$F_1({\mathfrak B})$ satisfies some identity $w_{(k,k)}^{(0)}$ and by Lemma \ref{consequences in one variable} satisfies the identity
$(y_1z_1)^k(y_1-z_1)^k=0$ in one variable. This means that $(y_1z_1)^k(y_1-z_1)^k=0$ in $F_1({\mathfrak B})$
which is impossible.
\end{proof}

The following theorem is the first main result of our paper.

\begin{theorem}\label{main theorem}
If $\mathfrak V$ is a proper subvariety of the variety $\mathfrak B$ of all bicommutative algebras such that
$\mathfrak V$ satisfies a polynomial identity $f=0$ of degree $k$, $0\not=f\in F({\mathfrak B})$, then
$c_n({\mathfrak V})$ is bounded by a polynomial of degree $k-1$.
\end{theorem}

\begin{proof}
Let
\begin{equation}\label{cocharacters of V}
\chi_n({\mathfrak V})=\sum_{\lambda\vdash n}m_{\lambda}({\mathfrak V})\chi_{\lambda},\quad n=1,2,\ldots,
\end{equation}
be the cocharacter sequence of $\mathfrak V$. By Proposition \ref{the n-th cocharacter of B} the summation in (\ref{cocharacters of V}) runs on
$\lambda=(\lambda_1,\lambda_2)\vdash n$. By Corollary \ref{consequences of identity of degree k} we obtain that
$m_{(\lambda_1,\lambda_2)}({\mathfrak V})=0$ for $\lambda_2\geq k$. If $\lambda_1-\lambda_2\leq 3k-1$, then
\[
m_{(\lambda_1,\lambda_2)}({\mathfrak V})\leq m_{(\lambda_1,\lambda_2)}({\mathfrak B}) \leq \lambda_1-\lambda_2+1\leq 3k.
\]
Now, let $\lambda_1-\lambda_2\geq 3k$. By Lemma \ref{consequences in one variable}, the variety $\mathfrak V$ satisfies the identities
\[
w_j=(y_1z_1)^k(y_1-z_1)^k(y_1z_2-y_2z_1)^{\lambda_2}y_1^jz_1^{\lambda_1-\lambda_2-3k-j}=0,\quad
j=0,1,\ldots,\lambda_1-\lambda_2-3k.
\]
All $w_j$, $j=0,1,\ldots,\lambda_1-\lambda_2-3k$, are linearly independent in $F_2({\mathfrak B})$
and are highest weight vectors for $\text{\rm GL}_2$-submodules of $F_2({\mathfrak B})$.
Hence the multiplicity $m_{\lambda}({\mathfrak V})$ satisfies the inequality
\[
m_{\lambda}({\mathfrak V})\leq m_{\lambda}({\mathfrak B})-(\lambda_1-\lambda_2-3k+1)=(\lambda_1-\lambda_2+1)-(\lambda_1-\lambda_2-3k+1)=3k.
\]
Hence (\ref{cocharacters of V}) satisfies the inequality
\[
\chi_n({\mathfrak V})=\sum_{(\lambda_1,\lambda_2)\vdash n\atop \lambda_2<k}m_{\lambda}({\mathfrak V})\chi_{(\lambda_1,\lambda_2)}
\leq \sum_{j=0}^{k-1}3k\chi_{(n-j,j)}.
\]
We obtain that the codimension sequence $c_n({\mathfrak B})$, $n=1,2,\ldots$, satisfies
\[
c_n({\mathfrak B})\leq\sum_{j=0}^{k-1}3kd_{(n-j,j)}
\]
which by Lemma \ref{degree for Sn-characters} is a polynomial of degree $k-1$.
\end{proof}

\begin{remark}\label{better bound}
We may precise Corollary \ref{consequences of identity of degree k}: If ${\mathfrak V}\subset{\mathfrak B}$
satisfies an identity $w_{\lambda}=0$ of degree $k$ and $\lambda_2>0$ in $\lambda=(\lambda_1,\lambda_2)\vdash k$,
then $\lambda_1\leq k-1$ and $\mathfrak V$
satisfies all identities $w_{\mu}^{(j)}=0$ for $\mu=(\mu_1,\mu_2)$ and $\mu_2\geq k-1$. Hence in this case $c_n({\mathfrak V})$
is bounded by a polynomial of degree $k-2$.
\end{remark}

\begin{example} The bound by a polynomial of degree $k-2$ in Remark \ref{better bound} is sharp.
Let $\mathfrak V$ be the subvariety of $\mathfrak B$ defined by the polynomial identity of right nilpotency $(\cdots((x_1x_2)x_3)\cdots )x_k=0$.
It is easy to see that the image in $G$ of the T-ideal $T({\mathfrak V})$ of the identities of $\mathfrak V$
is generated as an ordinary two-sided ideal by the products $y_{i_1}z_{i_2}\cdots z_{i_k}$. Hence if $\mu_2\geq k-1$, then
all $w_{(\mu_1,\mu_2)}^{(j)}$ belong to this T-ideal and
\[
w_{(n-k+2,k-2)}^{(n-2k+4)}=y_1^{n-2k+4}(y_1z_2-y_2z_1)^{k-2}
\]
does not belong to this ideal. Hence
$c_n({\mathfrak V})\geq d_{(n-k+2,k-2)}$ which is a polynomial of degree $k-2$. We do not know whether there exists
a variety ${\mathfrak V}\subset\mathfrak B$ satisfying a polynomial identity
in one variable of degree $k$ such that $c_n({\mathfrak V})$ grows as a polynomial of degree $k-1$.
\end{example}

\section{Two-dimensional algebras}

The classification of all two-dimensional algebras can be traced back to the two-dimensional part of the classification project
in the seminal book by B. Peirce \cite{Pe} published lithographically in 1870
in a small number of copies for distribution among his friends
and then reprinted posthumously in 1881 with addenda of his son C.S. Peirce.
(See Grattan-Guinness \cite{GG} for the contributions of Peirce.)
Starting in 2000 with the paper by Petersson \cite{P} (which contains also the history of the classification)
and the paper by Anan'in and Mironov \cite{AnM} there are several papers
containing different kinds of classification of two-dimensional algebras --
by Goze and Remm \cite{GR}, Ahmed, Bekbaev, and Rakhimov \cite{ABR},
Rausch de Traubenberg and Slupinski \cite{RTS}, Kaygorodov and Volkov \cite{KaV}. Concerning the polynomial identities of two-dimensional algebras,
Giambruno, Mishchenko, and Zaicev \cite{GMZ} proved that the growth of the codimension sequence $c_n(A)$ of such an algebra $A$
over a field of characteristic 0 is either linear (and bounded by $n+1$) or grows exponentially as $2^n$.

In this section we shall study the polynomial identities
of two-dimensional bicommutative algebras over an arbitrary field of characteristic 0.
It is well known that if $A$ is an algebra over an infinite field $K$ then the $K$-algebra $A$ and the $E$-algebra $E\otimes_KA$
have the same bases of polynomial identities for any extension $E$ of $K$.
More precisely, if $\{f_i\}\subset K\{X\}$ is a basis of the polynomial identities of the $K$-algebra $A$, then
$\{1\otimes f_i\}\subset E\otimes_KK\{X\}\cong E\{X\}$ is a basis of the polynomial identities of the $E$-algebra $E\otimes_KA$.
Hence in the sequel we may assume that the field $K$ is algebraically closed.

The classification of all two-dimensional bicommutative algebras over an arbitrary algebraically closed field of any characteristic is given
by Kaygorodov and Volkov in \cite{KaV}:

\begin{theorem}\label{classification by Kaygorodov and Volkov}
When the base field $K$ is algebraically closed any two-dimensional bicommutative algebra with nontrivial multiplication
is isomorphic to one of the seven algebras
\begin{equation}\label{seven algebras}
\{\text{\bf A}_3, \text{\bf B}_2(0), \text{\bf B}_2(1), \text{\bf D}_1(0,0), \text{\bf D}_2(1,1), \text{\bf D}_2(0,0), \text{\bf E}_1(0,0,0,0)\},
\end{equation}
where the algebras have bases $\{e_1,e_2\}$ and multiplication tables given below:
\[
\begin{array}{lllll}
\text{\bf A}_3:&e_1e_1=e_2,&e_1e_2=0,&e_2e_1=0,&e_2e_2=0;\\
\text{\bf B}_2(0):&e_1e_1=0,&e_1e_2=e_1,&e_2e_1=0,&e_2e_2=0;\\
\text{\bf B}_2(1):&e_1e_1=0,&e_1e_2=0,&e_2e_1=e_1,&e_2e_2=0;\\
\text{\bf D}_1(0,0):&e_1e_1=e_1,&e_1e_2=e_1,&e_2e_1=0,&e_2e_2=0;\\
\text{\bf D}_2(1,1):&e_1e_1=e_1,&e_1e_2=e_2,&e_2e_1=e_2,&e_2e_2=0;\\
\text{\bf D}_2(0,0):&e_1e_1=e_1,&e_1e_2=0,&e_2e_1=0,&e_2e_2=0;\\
\text{\bf E}_1(0,0,0,0):&e_1e_1=e_1,&e_1e_2=0,&e_2e_1=0,&e_2e_2=e_2.\\
\end{array}
\]
\end{theorem}

\begin{remark}
In the classification of two-dimensional bicommutative algebras given in the preliminary version of \cite[Section 7.2]{KaV}
there is one more algebra $\text{\bf D}_1(1,0)$. This algebra is isomorphic to
the algebra $\text{\bf D}_1(0,0)$ because in \cite[Table 1]{KaV}
the pairs $(0,0)$ and $(1,0)$ belong to the same orbit of the cyclic group of order 2 generated by $\varrho$, where
\[
^{\varrho}(\alpha,\beta)=(1-\alpha+\beta,\beta),\quad (\alpha,\beta)\in K^2.
\]
\end{remark}

The algebra $\text{\bf A}_3$ is commutative and nilpotent of class 3. Hence the T-ideal of its polynomial identities is generated by
\[
x_1x_2=x_2x_1,\quad (x_1x_2)x_3=0
\]
and the cocharacter sequence of $\text{\bf A}_3$ is
\[
\chi_1(\text{\bf A}_3)=\chi_{(1)},\quad\chi_2(\text{\bf A}_3)=\chi_{(2)},\quad \chi_n(\text{\bf A}_3)=0,\quad n=3,4,\ldots.
\]
Similarly, the algebras $\text{\bf D}_2(1,1), \text{\bf D}_2(0,0)$, and $\text{\bf E}_1(0,0,0,0)$ are associative-commutative,
have the same bases of polynomial identities consisting of
\[
x_1x_2=x_2x_1,\quad (x_1x_2)x_3=x_1(x_2x_3)
\]
and their cocharacter sequence is
\[
c_n(\text{\bf D}_2(1,1))=c_n(\text{\bf D}_2(0,0))=c_n(\text{\bf E}_1(0,0,0,0))=\chi_{(n)},\quad n=1,2,\ldots .
\]
Hence to complete the description of the polynomial identities of two-dimensional bicommutative algebras
it is sufficient to handle the cases $A=\text{\bf B}_2(0),\text{\bf B}_2(1),\text{\bf D}_1(0,0)$.
These three algebras satisfy the condition $\dim A^2=1$. For our purposes it is more convenient to have another presentation of the algebras.
The next proposition shows that over any field of characteristic 0 the two-dimensional bicommutative algebras $A$ with $\dim A^2=1$
are in the list of Theorem \ref{classification by Kaygorodov and Volkov}.

\begin{proposition}\label{classification of two-dimensional algebras}
Over an arbitrary field $K$ of characteristic $0$ there are only five nonisomorphic two-dimensional bicommutative algebras $A$ such that
$\dim A^2=1$. They are one-generated and isomorphic to the algebras
\begin{equation}\label{five algebras}
A_{0,0},A_{1,1},A_{0,1},A_{1,0},A_{1,-1}
\end{equation}
where the algebra $A_{\pi,\varrho}$ is with multiplication given in {\rm (\ref{two-dimensional algebras})}.
The five algebras in {\rm (\ref{five algebras})} are isomorphic, respectively, to the algebras
\[
\text{\bf A}_3, \text{\bf D}_2(0,0), \text{\bf B}_2(0), \text{\bf B}_2(1), \text{\bf D}_1(0,0)
\]
from the list in {\rm (\ref{seven algebras})}.
\end{proposition}

\begin{proof}
Let $A$ have a basis $\{a,b\}$, where $a\in A\setminus A^2$, $b\in A^2$. If $a^2\not=0$, then $a^2=\alpha b$, $0\not=\alpha\in K$.
Hence $A$ is generated by $a$ and has a basis $\{a,a^2\}$.

Now, let $a^2=0$. If $ab=\beta b\not=0$, $0\not=\beta\in K$, then the identity of right-commutativity gives
\[
\beta ba=(ab)a=(aa)b=0,
\]
and hence $ba=0$. Let $b^2=\gamma b$, $\gamma\in K$. Then for $\eta\in K$
\[
(a+\eta b)^2=\eta(\beta+\eta\gamma)b
\]
and we always can choose $\eta$ in a way to have $(a+\eta b)^2\not=0$. Again $A$ is one generated and has a basis $\{a+\eta b,(a+\eta b)^2\}$.
Similarly, if $ba\not=0$, then $ab=0$ and again $A$ is one-generated.

Hence we may assume that $A$ has a basis $\{r,r^2\}$. Let
\[
rr^2=\pi r^2, r^2r=\varrho r^2, \quad \pi,\varrho\in K.
\]
Then the right-commutativity implies
\[
r^2r^2=(rr)r^2=(rr^2)r=\pi r^2r=\pi\varrho r^2.
\]
Hence the multiplication of the algebra $A$ is as of the algebra $A_{\phi,\varrho}$ in (\ref{two-dimensional algebras}).
Let $\pi=0$, $\varrho\not=0$. If we replace the generator $r$ by $r=\varrho r_1$, then
\[
r^2r=\varrho r^2,\quad \varrho^3r_1^2r_1=\varrho \varrho^2r_1^2,\quad r_1^2r_1=r_1^2,
\]
i.e., $A_{0,\varrho}\cong A_{0,1}$. Similarly, $A_{\pi,0}\cong A_{1,0}$. If $\pi=\varrho\not=0$, then the change of the generator $r$ with
$r=\pi r_1$ gives that
\[
r^2r=\pi r^2,\quad \pi^3r_1^2r_1=\pi \pi^2r_1^2,\quad r_1^2r_1=r_1^2,\quad r_1r_1^2=r_1^2,
\]
and $A_{\pi,\pi}\cong A_{1,1}$. Finally, let $\pi\not=\varrho$ be different from 0. We fix solutions $\xi$ and $\eta$ of the linear system
\[
\pi(\xi+\varrho\eta)=1,\quad \varrho(\xi+\pi\eta)=-1.
\]
Then $r_1=\xi r+\eta r^2$ satisfies the conditions
\[
r_1^2=(\xi+\pi\eta)(\xi+\varrho\eta)r^2=-\frac{1}{\pi\varrho}r^2,
\]
\[
r_1r_1^2=-\frac{1}{\pi\varrho}(\xi r+\eta r^2)r^2=-\frac{1}{\pi\varrho}\pi(\xi+\eta\varrho)r^2=\pi(\xi+\eta\varrho)r_1^2=r_1^2,
\]
\[
r_1^2r_1=\varrho(\xi+\pi\eta)r_1^2=-r_1^2,
\]
i.e., $A_{\pi,\varrho}\cong A_{1,-1}$.

The isomorphisms between the algebras
$A_{0,0},A_{1,1},A_{0,1},A_{1,0},A_{1,-1}$ and, respectively, the algebras
$\text{\bf A}_3, \text{\bf D}_2(0,0), \text{\bf B}_2(0), \text{\bf B}_2(1), \text{\bf D}_1(0,0)$
are given as follows:
\[
\begin{array}{lll}
A_{0,0}\cong\text{\bf A}_3:&r\to e_1,&r^2\to e_2;\\
A_{1,1}\cong\text{\bf D}_2(0,0):&r\to e_1+e_2,&r^2\to e_1;\\
A_{0,1}\cong\text{\bf B}_2(0):&r\to e_1+e_2,&r^2\to e_1;\\
A_{1,0}\cong\text{\bf B}_2(1):&r\to e_1+e_2,&r^2\to e_1;\\
A_{1,-1}\cong\text{\bf D}_1(0,0):&r\to e_1-2e_2,&r^2\to-e_1.\\
\end{array}
\]
\end{proof}

The next theorem gives bases for the polynomial identities and the cocharacter sequences of the three nonassociative algebras $A_{0,1},A_{1,0},A_{1,-1}$.

\begin{theorem}\label{PIs of two-dimensional algebras}
{\rm (i)} As subvarieties of the variety $\mathfrak B$ of all bicommutative algebras the varieties $\text{\rm var}(A_{0,1})$ and $\text{\rm var}(A_{1,0})$ generated
by the algebras $A_{0,1}$ and $A_{1,0}$ are defined by the identities of left-nilpotency $x_1(x_2x_3)=0$ and right-nilpotency $(x_1x_2)x_3=0$, respectively.
Their cocharacter and codimension sequences coincide and are
\[
\chi_1(A_{0,1})=\chi_1(A_{1,0})=\chi_{(1)},\chi_n(A_{0,1})=\chi_{(n)}+\chi_{(n-1,1)},\quad n=2,3,\ldots,
\]
\[
c_n(A_{0,1})=c_n(A_{1,0})=n,\quad n=1,2,\ldots.
\]
 {\rm (ii)} The algebra $A_{1,-1}$ generates the whole variety $\mathfrak B$.
\end{theorem}

\begin{proof}
(i) Clearly the algebra $A_{0,1}$ satisfies the polynomial identity $x_1(x_2x_3)=0$. The origins in $F=F({\mathfrak B})$ of the polynomials
$w_{\lambda}^{(j)}$ from (\ref{hwv of W(lambda)}) have the form
\[
w_{(n)}^{(j)}(x_1)=\underbrace{x_1(\cdots (x_1(((x_1}_{j\text{ times}}\underbrace{x_1)\cdots )x_1}_{n-j\text{ times}}))\cdots),
\]
\[
w_{(\lambda_1,\lambda_2)}^{(j)}(x_1,x_2)
=\underbrace{x_1(\cdots x_1}_{j\text{ times}}((\cdots((x_1x_2-x_2x_1)^{\lambda_2}\underbrace{x_1)\cdots )x_1}_{\lambda_1-\lambda_2-j\text{ times}})\cdots).
\]
Obviously $w_{(\lambda_1,\lambda_2)}^{(j)}$ follows from $x_1(x_2x_3)=0$ for $\lambda=(n)$, $j=2,\ldots,n-1$, $n\geq 3$,
for $\lambda=(n-1,1)$, $j=1,\ldots,n-2$, and for $\lambda=(\lambda_1,\lambda_2)$, $\lambda_2\geq 2$.
On the other hand $w_{(n)}^{(1)}(r)=r^2\not=0$, $w_{(n-1,1)}^{(0)}(r,r^2)=-r^2\not=0$. This shows that the identities of $A_{0,1}$ follow from $x_1(x_2x_3)=0$,
$\chi_1(A_{0,1})=\chi_{(1)}$, $\chi_n(A_{0,1})=\chi_{(n)}+\chi_{(n-1,1)}$, $n=2,3,\ldots$, and $c_n(A_{0,1})=n$, $n=1,2,\ldots$. The proof for $A_{1,0}$ is similar.

(ii) By Corollary \ref{generating by free algebra of rank 1} it is sufficient to show that the algebra $A_{1,-1}$ does not satisfy an identity in one variable.
Let
\[
w_{(n)}(y_1,z_1)=\sum_{j=1}^{n-1}\xi_jw_{(n)}^{(j)}(y_1,z_1),\quad \xi_j\in K,
\]
be a polynomial in $G$ which corresponds to a homogeneous polynomial identity $f(x_1)=0$ in one variable and of degree $n\geq 2$,
$0\not=f(x_1)\in F({\mathfrak B})$. We shall evaluate $f(x_1)$ on all $\gamma r+\delta r^2\in A_{1,-1}$, $\gamma,\delta\in K$.
Since
\[
(\gamma r+\delta r^2)^2=(\gamma^2-\delta^2)r^2,
\]
\[
(\gamma r+\delta r^2)\cdot(\gamma r+\delta r^2)^2=(\gamma-\delta)(\gamma^2-\delta^2)r^2,
\]
\[
(\gamma r+\delta r^2)^2\cdot(\gamma r+\delta r^2)=-(\gamma+\delta)(\gamma^2-\delta^2)r^2,
\]
we obtain that the evaluation of the proimage of $w_{(n)}^{(j)}(y_1,z_1)$ on $\gamma r+\delta r^2$ is equal to
\[
(-1)^{n-j-1}(\gamma^2-\delta^2)(\gamma-\delta)^{j-1}(\gamma+\delta)^{n-j-1}r^2=(-1)^{n-1}(\delta-\gamma)^j(\delta+\gamma)^{n-j}.
\]
Hence
\[
f(\gamma r+\delta r^2)=(-1)^{n-1}w_{(n)}(\delta-\gamma,\delta+\gamma)r^2=0.
\]
When $\gamma$ and $\delta$ run on the whole field $K$ the same holds for $\delta-\gamma$ and $\delta+\gamma$. Therefore the polynomial
$w_{(n)}(y_1,z_1)$ vanishes evaluated on the infinite field $K$ and hence is identically equal to 0. This means that
$A_{1,-1}$ does not satisfy any polynomial identity in one variable and hence generates the whole variety $\mathfrak B$.
\end{proof}

The following easy lemma gives an upper bound for the codimensions of a finite dimensional algebra.
It makes more precise the bound for the codimensions established for graded algebras in \cite{BaD} and independently in \cite{GZ4}.

\begin{lemma}\label{growth of codimensions}
Let $A$ be a finite dimensional algebra and let $k$ be a positive integer. Then for all $n\geq k$
\[
c_n(A)\leq \dim(A^k)\dim^n(A).
\]
\end{lemma}

\begin{proof}
Let $\dim(A)=p$ and $\dim(A^k)=q$. We fix a basis $\{r_1,\ldots,r_q\}$ of $A^k$ and extend it to a basis $\{r_1,\ldots,r_p\}$ of $A$.
We consider the multilinear identity
\[
f(x_1,\ldots,x_n)=\sum_{(\sigma)}\xi_{(\sigma)}(x_{\sigma(1)}\cdots)(\cdots x_{\sigma(n)})=0,\quad \xi_{(\sigma)}\in K,
\]
where the summation runs on all permutations $\sigma\in S_n$ and all possible bracket decompositions. Clearly, $f(x_1,\ldots,x_n)=0$
is a polynomial identity for $A$ if and only if $f(r_{i_1},\ldots,r_{i_n})=0$ for all possible choices of the basis elements $r_{i_1},\ldots,r_{i_n}$.
Since $\deg(f)=n$ and $n\geq k$ the evaluations of $f(x_1,\ldots,x_n)$ on $R$ belong to $A^k$. Let
\[
f(r_{i_1},\ldots,r_{i_n})=\sum_{j=1}^qf_j(r_{i_1},\ldots,r_{i_n})r_j,
\]
where $f_j(r_{i_1},\ldots,r_{i_n})\in K$ are linear functions in the coefficients $\xi_{(\sigma)}$.
Considering $\xi_{(\sigma)}$ as unknowns, we obtain the linear homogeneous system
\begin{equation}\label{linear system}
f_j(r_{i_1},\ldots,r_{i_n})=0,\quad r_{i_1},\ldots,r_{i_n}\in\{r_1,\ldots,r_p\},j=1,\ldots,q.
\end{equation}
The system has $n!C_n$ unknowns, where $C_n$ is the $n$-th Catalan number (equal to the number of the bracket decompositions).
Since the codimension $c_n(A)$ is equal to the rank of the system and the system has $qp^n$ equations, its rank is less or equal to $qp^n$
and the same holds for the $n$-th codimension $c_n(A)$.
\end{proof}

\begin{remark}
It was shown in \cite{GMZ} that if the two-dimensional algebra $A$ has a one-dimensional nilpotent ideal, then
$c_n(A)\leq n+1$. The algebras $A_{0,1}$ and $A_{1,0}$ satisfy this condition and Theorem \ref{PIs of two-dimensional algebras} (i) shows that
their codimensions are very close to the upper bound. For the algebra $A_{1,-1}$ the results in \cite{GMZ} give that
\[
\frac{2^n}{n^2}\leq c_n(A_{1,-1})\leq 2^{n+1}.
\]
Since $\dim(A_{1,-1})=2$ and $\dim(A_{1,-1}^2)=1$
Lemma \ref{growth of codimensions} implies $c_n(A_{1,-1})\leq 2^n$.
By \cite{DIT} and Theorem \ref{PIs of two-dimensional algebras} (ii) we have that $c_n(A_{1,-1})=c_n({\mathfrak B})=2^n-2$.
Again, this is very close to the upper bound $2^n$.
\end{remark}

\end{document}